\documentclass[12pt]{article}
\usepackage[numbers,sort&compress]{natbib}
\usepackage{amssymb,amsmath,mathrsfs,amsfonts,amsthm}
\usepackage{enumerate}
\usepackage{paralist}
\usepackage{color}
\usepackage{epsfig}
\usepackage{graphicx}
\usepackage{setspace}
\usepackage{appendix}
\begin{document}
 \def\pd#1#2{\frac{\partial#1}{\partial#2}}
\def\dfrac{\displaystyle\frac}
\let\oldsection\section
\renewcommand\section{\setcounter{equation}{0}\oldsection}
\renewcommand\thesection{\arabic{section}}
\renewcommand\theequation{\thesection.\arabic{equation}}

\newtheorem{thm}{Theorem}[section]
\newtheorem{cor}[thm]{Corollary}
\newtheorem{lem}[thm]{Lemma}
\newtheorem{prop}[thm]{Proposition}
\newtheorem*{con}{Conjucture A}
\newtheorem*{questionA}{Question}
\newtheorem*{thmB}{Theorem B}
\newtheorem{remark}{Remark}[section]
\newtheorem{definition}{Definition}[section]

\title{Global dynamics of competition models with  nonsymmetric nonlocal dispersals when one diffusion rate is small
\thanks{The first author is supported by Chinese NSF (No. 11501207).
The second author is  supported by  NSF  of China (No. 11431005), NSF of Shanghai  (No. 16ZR1409600).} }

\author{Xueli Bai{\thanks{E-mail: mybxl110@163.com}}\\ {Department of Applied Mathematics, Northwestern Polytechnical
University,}\\{\small 127 Youyi Road(West), Beilin 710072,
Xi'an, P. R. China.}\\ Fang Li{\thanks{Corresponding author.
E-mail: lifang55@mail.sysu.edu.cn}}\\ {School of Mathematics, Sun Yat-sen University,}\\{\small No. 135, Xingang Xi Road, Guangzhou, 510275, P. R. China} }

\date{}
\maketitle{}

\begin{abstract}
In this paper, we study the global dynamics of a general $2\times 2$ competition models with nonsymmetric nonlocal dispersal operators.   Our results indicate that local stability implies global stability provided that one of the diffusion rates is sufficiently small. This paper continues the work in \cite{BaiLi2017}, where competition models with symmetric nonlocal operators are  considered.
\end{abstract}

{\bf Keywords}: nonlocal dispersal,competitive Lotka-Volterra system, local stability, global convergence
\vskip3mm {\bf MSC (2010)}: Primary: 45G15, 45M05; Secondary: 45M10, 45M20.


\section{Introduction}
In this paper, we mainly consider $2\times 2$ reaction diffusion systems with nonlocal diffusion of the following type:
\begin{equation}\label{original}
\begin{cases}
u_t= d \mathcal{K}[u]  +u f(x,u,v) &\textrm{in } \Omega\times[0,\infty),\\
v_t= D \mathcal{P}[v]  +v g(x,u,v) &\textrm{in } \Omega\times[0,\infty),\\
u(x,0)=u_0(x),~v(x,0)=v_0(x)  &\textrm{in } \Omega,
\end{cases}
\end{equation}
where $d,D>0$, $\Omega$ is a  bounded domain in $\mathbb R^n$, $n\geq 1$, $\mathcal{K}$ and $\mathcal{P}$ represent nonlocal operators, which will be defined later.

When nonlocal operators are replaced by differential operators, two-component systems like (\ref{original}) allow for a large range of possible phenomena in chemistry, biology, ecology, physics and so on and have been extensively studied.  One of the most famous phenomena is the idea, which was first proposed by Alan Turing,  that a   stable state in the local system can become unstable in the presence of diffusion. This remarkable idea is called ``diffusion-driven instability", which is one of the most classical theories in the studies of pattern formations. Another important phenomena comes from ecology, where random diffusion is introduced to model dispersal strategies \cite{Skellam} and there are tremendous studies in this direction, see the books \cite{CC-book, OL-book}.  Indeed, dispersal strategies of organisms have been a central topic in ecology.
However, when a long range dispersal is considered, nonlocal reaction diffusion equations are commonly used \cite{Fife,Grinfeld,HMMV,Lutscher,Mogilner-E,Turchin}, where the nonlocal operator takes the following form:
$$
\mathcal{L}u :=\int_\Omega k(x,y)u(y)dy-a(x)u(x).
$$

To be more specific, this paper is motivated by the studies of
Lotka-Volterra type weak competition models with spatial inhomogeneity, that is $f(x,u,v)=m(x)- u- c v$, $g(x,u,v)= m(x)- bu-  v$
with $0< bc<1$ in (\ref{original}).
For this case,  $u(t,x)$, $v(t,x)$ are the population densities of two competing species, $d, D>0$ are their dispersal rates, which measure the
total number of dispersal individuals per unit time, respectively, while $m(x)$ is nonconstant and  represents spatial distribution of resources.
This type of models reflects the interactions among dispersal strategies, spatial heterogeneity of resources and interspecific competition abilities on the persistence of species and has received extensive studies from both mathematicians and ecologists for the last three decades. For models with random diffusion, see  \cite{CC98,  HeNi, LamNi, LWW11, Lou}  and the references therein, while for  models with nonlocal dispersals, see \cite{BaiLi2015, BaiLi2017, HNShen2012, LLW14} and the references therein.

Inspired by the nature of this type of models, in \cite{Lou},   an insightful conjecture was proposed and partially verified:
\begin{con}
The locally stable steady state  is globally asymptotically stable.
\end{con}

It is known that this conjecture is true for ODE systems. Recently, for symmetric PDE case,  this conjecture  has been completely resolved in \cite{HeNi}. Moreover, if random diffusion is replaced by symmetric nonlocal operators, this conjecture is also verified in \cite{BaiLi2017}. In the proofs of these results, the symmetry property of operators and the particular form of reaction terms  are crucial.  These naturally  lead us to investigate the system (\ref{original}) with nonsymmetric operators and more  general reaction terms.


Now let us designate the definitions of nonlocal operators in (\ref{original}) and assumptions imposed on $f,g$.
Denote
$$\mathbb X= C(\bar\Omega),\ \mathbb  X_+ =\{ u\in \mathbb X \ | \ u\geq 0\},\ \mathbb  X_{++}=\mathbb  X_+\setminus \{  0\}.
$$
For $\phi\in \mathbb X$, define
\begin{itemize}
\item[\textbf{(D)}] $ \displaystyle\mathcal{K} [\phi] = \int_{\Omega}k(x,y)\phi(y)dy- \phi(x), \  \mathcal{P} [\phi] = \int_{\Omega}p(x,y)\phi(y)dy-  \phi(x),$
\item[\textbf{(N)}] $ \displaystyle\mathcal{K}[\phi] = \int_{\Omega}k(x,y)\phi(y)dy- \int_{\Omega}k(y,x)dy \phi(x), \  \mathcal{P}[\phi] = \int_{\Omega}p(x,y)\phi(y)dy- \int_{\Omega}p(y,x)dy \phi(x)$,
\end{itemize}
where the kernels $k(x, y)$, $p(x,y)$ describe the rate at which organisms move from point $y$ to point $x$. Here the operators  defined in \textbf{(D)} and \textbf{(N)} correspond to nonlocal operator with lethal boundary condition and  no flux boundary condition  respectively.  See \cite{HMMV} for the derivation of  different types of nonlocal operators.

Throughout this paper, unless designated otherwise, we assume that
\begin{itemize}
\item[\textbf{(A0)}]  $k(x,y)$, $p(x,y)\in C(\mathbb R^n\times \mathbb R^n)$ are nonnegative. $k(x,x)>0$, $p(x,x)>0$ in $\mathbb R^n$. Moreover,  $\int_{\mathbb R^n} k(x,y)dy= \int_{\mathbb R^n} k(y,x)dy= \int_{\mathbb R^n} p(x,y)dy= \int_{\mathbb R^n} p(y,x)dy=1$.
\item[\textbf{(A1)}] $f(x,u,v), g(x,u,v), f_u(x,u,v), f_v(x,u,v), g_u(x,u,v), g_v(x,u,v)\in C( \bar\Omega \times \mathbb R_+ \times \mathbb R_+)$.
\item[\textbf{(A2)}] $f_u<0, g_v<0$ in $\bar\Omega \times \mathbb R_+ \times \mathbb R_+$.
\item[\textbf{(A3)}] $f_v<0, g_u<0$ in $\bar\Omega \times \mathbb R_+ \times \mathbb R_+$.
\item[\textbf{(A4)}] There exists $M>0$ such that $f(x,u,v)<0$ for $u\geq M$,  $g(x,u,v)<0$ for $v\geq M$.
\item[\textbf{(A5)}] $f_v g_u < f_u g_v$ in $\bar\Omega \times \mathbb R_+ \times \mathbb R_+$.
\end{itemize}

Obviously, the reaction terms $f(x,u,v)=m(x)- u- c v$, $g(x,u,v)= m(x)- bu-  v$  satisfy \textbf{(A2)}-\textbf{(A4)} provided that $m\in L^{\infty}(\Omega)$  and the assumption \textbf{(A5)} corresponds to $bc<1$. Moreover, different from PDE case, for models with nonlocal operators,  the optimal regularity of solutions is at most the same as the regularity of the reaction terms. Hence, \textbf{(A1)} is imposed  to guarantee that the solutions could be continuous in space variable.

To better demonstrate our main results and techniques, some explanations are in place. Let $(U(x),V(x))$  denote a nonnegative steady state of (\ref{original}),  then   there are {\it at most} three possibilities:
\begin{itemize}
\item $(U,V)= (0,0)$ is called a {\it trivial steady state}.
\item $(U,V)=(\theta_d, 0)$ or $(U,V)=(0, \eta_D)$ is called a {\it semi-trivial steady state}, where $\theta_d$, $\eta_D$ are the positive solutions to single-species models
    \begin{equation}\label{singled}
    d \mathcal{K}[U]  +U f(x, U, 0)=0,
    \end{equation}
    and
    \begin{equation}\label{singleD}
     D \mathcal{P}[V]  +V g(x,0, V)=0
    \end{equation}
    respectively.
\item $U>0,\ V>0$, and we call $(U,V)$ a {\it coexistence/positive steady state}.
\end{itemize}

{\it The  main result} in this paper  gives a  classification of  the global dynamics to the competition system (\ref{original})  provided that one diffusion rate is small.

\begin{thm}\label{thm-main}
Assume that \textbf{(A0)}-\textbf{(A5)} hold. Also assume that (\ref{original}) admits two semi-trivial steady states $(\theta_d, 0)$ and $(0,\eta_D)$. Then for the global dynamics of the system (\ref{original}) with nonlocal operators defined in \textbf{(D)} or \textbf{(N)}, we have the following statements provided that $d$ is sufficiently small:
\begin{itemize}
\item[(i)] If $\mu_0>0$, $\nu_D^0>0$ and in addition $k(x,y)>0$ for $x, y\in \bar\Omega$, then the system (\ref{original}) admits a unique positive steady state in $\mathbb X \times \mathbb X$, which is globally asymptotically stable relative to $\mathbb X_{++} \times \mathbb X_{++}$;
\item[(ii)] If $\mu_0>0$ and $\nu_D^0<0$, then $(0,\eta_D)$ is globally asymptotically stable relative to $\mathbb X_{++} \times \mathbb X_{++}$;
\item[(iii)] If $\mu_0<0$, then $(\theta_d, 0)$ is globally asymptotically stable relative to $\mathbb X_{++} \times \mathbb X_{++}$.
 \end{itemize}
Here $\mu_0$ and $\nu_D^0$ are defined in (\ref{mu0nu0}).
\end{thm}

Based on the definitions of $\mu_0$ and $\nu_D^0$ and Lemma \ref{lm-d=0}, indeed Theorem \ref{thm-main} verifies \textbf{Conjecture A} for more general  reaction-diffusion systems provided that one diffusion rate is small.

In this paper, we only  demonstrate the proof of Theorem \ref{thm-main} for the nonlocal operators defined in \textbf{(D)}, since the proof for nonlocal operators defined in  \textbf{(N)} is  almost the same.

It seems that the extra condition that $k(x,y)>0$ for $x, y\in \bar\Omega$ in Theorem \ref{thm-main}(i) is imposed for technical reasons. However, it remains unknown that whether the complexity of nonsymmetric operators could result in multiple positive steady states under the assumption \textbf{(A5)}.  We will return to this topic in future work.

This paper is organized as follows. In Section 2, some useful properties related to single equations, nonlocal eigenvalue problems and monotone systems are prepared. Then the limiting system as $d\rightarrow 0^+$ is investigated in Section 3. At the end, Section 4 is devoted to the proof of Theorem \ref{thm-main}.

\section{Preliminary}

\subsection{Properties of single equations}

It is known that under assumptions \textbf{(A0)}, \textbf{(A1)}, \textbf{(A2)} and \textbf{(A4)}, the single equation (\ref{singled}) admits a unique positive solution in $\mathbb X$, denoted by $\theta_d$, if and only if
$$
\sup \left\{\textrm{Re}\, \lambda\, |\, \lambda\in \sigma(d\mathcal{K}+ f(x,0,0))   \right\}>0,
$$
while the single equation (\ref{singleD}) admits a unique positive solution in $\mathbb X$, denoted by $\eta_D$, if and only if
$$
\sup \left\{\textrm{Re}\, \lambda\, |\, \lambda\in \sigma(D\mathcal{P}+ g(x,0,0))   \right\}>0.
$$
See  \cite[Theorem 2.1]{BaiLi2017}  for details.

For further analysis, we need estimate the asymptotic behavior of $\theta_d$ as $d\rightarrow 0^+$.
Notice that due to \textbf{(A2)} and implicit function theorem,   there exists $u= F(x,v)\in C(\bar\Omega\times \mathbb R_+)$ with $ F_v(x,v)\in C(\bar\Omega\times \mathbb R_+)$
such that
$$
f(x,F(x,v), v)=0 \  \ \textrm{in} \ \bar\Omega\times \mathbb R_+,
$$
and
\begin{equation}\label{ift-F}
F_v(x,v) = - {f_v(x,u, v) \over f_u (x,u, v)} \  \ \textrm{in} \ \bar\Omega\times \mathbb R_+.
\end{equation}
Also  denote $F_+(x,v) =\max \{ 0, F(x,v) \}$ in $\bar\Omega\times \mathbb R_+$.

\begin{prop}\label{prop-d=0-theta}
There exists $C_1>0$ such that
$$
\left(F_+(x,0) -C_1 d \right)_+ < \theta_d < F_+(x,0) + C_1 \sqrt{d} \ \ \textrm{in} \ \bar\Omega
$$
for $0<d<1$.
\end{prop}

\begin{proof}
First of all, by \textbf{(A4)} and strong maximum principle, it is easy to show that $0< \theta_d < M$.
Then due to  \textbf{(A2)}, there exists $c_0$ such that
$$
f_u(x,u,0) \leq -c_0\ \ \textrm{for}\  x\in \bar\Omega, u\in [0,M].
$$

On the one side, set $\hat u = F_+(x,0) + c_1 \sqrt{d}$, $c_1>0$. We compute as follows
\begin{eqnarray*}
&& d \mathcal{K}[\hat u] + \hat u f(x, \hat u, 0)\\
&=& d\int_{\Omega} k(x,y) ( F_+(y,0)  + c_1 \sqrt{d}  ) dy - d ( F_+(x,0) + c_1 \sqrt{d}  )\\
&& +  ( F_+(x,0) + c_1 \sqrt{d}  ) f(x, F_+(x,0) + c_1 \sqrt{d}, 0)\\
&\leq & d\int_{\Omega} k(x,y) ( F_+(y,0)  + c_1 \sqrt{d}  ) dy\\
&& +  ( F_+(x,0) + c_1 \sqrt{d}  )\left(  f(x, F_+(x,0) + c_1 \sqrt{d}, 0) - f(x, F_+(x,0) , 0) \right)\\
&\leq & d\int_{\Omega} k(x,y) ( F_+(y,0)  + c_1 \sqrt{d}  ) dy  - c_0c_1 \sqrt{d}( F_+(x,0) + c_1 \sqrt{d}  )\\
&<& 0,
\end{eqnarray*}
provided that $c_1$ is large enough.

On the other side, set $\underline{u} = \left(F_+(x,0) -c_2 d \right)_+$. Then compute as follows
\begin{eqnarray*}
&& d \mathcal{K}[\underline u] + \underline u f(x, \underline u, 0)\\
&\geq &   - d \left(F_+(x,0) -c_2 d \right)_+  +  \left(F_+(x,0) -c_2 d \right)_+ f(x, \left(F_+(x,0) -c_2 d \right)_+ , 0)\\
&=  &  - d \left(F_+(x,0) -c_2 d \right)_+\\
&& + \left(F_+(x,0) -c_2 d \right)_+ \left( f(x,  F_+(x,0) -c_2 d  , 0) - f(x,  F_+(x,0) , 0)\right)\\
&\geq &  \left(F_+(x,0) -c_2 d \right)_+ \left( -d + c_0 c_2 d \right)\\
&\geq & 0,
\end{eqnarray*}
provided that $c_2$ is large enough.

At the end, choose $C_1 = \max \{ c_1, c_2 \}$ and the desired conclusion follows.
\end{proof}

\subsection{Properties of nonlocal eigenvalue problems}

First of all,  the linearized operator of (\ref{original})  at $(\theta_d, 0)$ is
\begin{equation}\label{lin-d}
\mathcal{L}_{(\theta_d,0)} {\phi\choose\psi}={d\mathcal{K}[\phi]+[f(x, \theta_d, 0)+ \theta_d f_u(x, \theta_d, 0)]\phi + \theta_d f_v(x, \theta_d, 0)\psi \choose D\mathcal{P}[\psi]+ g (x, \theta_d, 0)\psi}.
\end{equation}
Also,  the linearized operator of (\ref{original})  at $(0, \eta_D)$ is
\begin{equation}\label{lin-D}
\mathcal{L}_{(0, \eta_D)} {\phi\choose\psi}={ d \mathcal{K}[\phi] + f(x,0,\eta_D)\phi  \choose D\mathcal{P}[\psi] +[g(x,0,\eta_D) + \eta_D g_v(x,0,\eta_D)]\psi + \eta_D g_u(x,0,\eta_D) \phi }.
\end{equation}
Denote
\begin{eqnarray}\label{PEV}
&&  \mu_{(\theta_d,0)}=\sup \left\{\textrm{Re}\, \lambda\, |\, \lambda\in \sigma(D\mathcal{P}+ g (x, \theta_d, 0))   \right\}, \\
&&  \nu_{(0, \eta_D)}= \sup \left\{\textrm{Re}\, \lambda\, |\, \lambda\in \sigma( d \mathcal{K}  +f(x,0,\eta_D))   \right\}.\nonumber
\end{eqnarray}
It is known that the signs of $\mu_{(\theta_d,0)}$ and $\nu_{(0, \eta_D)}$ determine the local stability/instability of
$(\theta_d,0)$ and  $(0, \eta_D)$ respectively.  This is explicitly stated as follows and the proof is omitted since it is standard.
\begin{lem}\label{lm-signs}
Assume that the assumptions \textbf{(A0)}, \textbf{(A1)} hold. Then
\begin{itemize}
\item[(i)] $(\theta_d,0)$ is locally unstable if $\mu_{(\theta_d,0)}>0$; $(\theta_d,0)$ is locally stable if $\mu_{(\theta_d,0)}<0$; $(\theta_d,0)$ is neutrally stable if $\mu_{(\theta_d,0)}=0$.
\item[(ii)] $(0, \eta_D)$ is locally unstable if $\nu_{(0, \eta_D)}>0$; $(0, \eta_D)$ is locally stable if $\nu_{(0, \eta_D)}<0$; $(0, \eta_D)$ is neutrally stable if $\nu_{(0, \eta_D)}=0$.
\end{itemize}
\end{lem}

The characterization of $\mu_{(\theta_d,0)}$ and  $\nu_{(0, \eta_D)}$    plays  an important role in this paper. We present the related results from \cite{LiCovilleWang}  as follows for the convenience of readers.

\begin{thm}[\cite{LiCovilleWang}]\label{thm-eigenv-general}
Assume that $\gamma \in C(\bar\Omega)$ and let
$$
\lambda_1(d) = \sup \left\{\textrm{Re}\, \lambda\, |\, \lambda\in \sigma( d \mathcal{K}  + \gamma(x))   \right\}.
$$
Then we have
\begin{equation}\label{eigenv-Wang}
\lambda_1(d)= \sup_{\phi\in C(\bar\Omega),\, \phi>0 \ \textrm{on} \ \bar\Omega} \inf_{x\in\Omega} {d \mathcal{K}[\phi]  + \gamma \phi \over \phi }
=\inf_{\phi\in C(\bar\Omega),\, \phi>0 \ \textrm{on} \ \bar\Omega} \sup_{x\in\Omega} {d \mathcal{K}[\phi]  + \gamma \phi \over \phi }.
\end{equation}
In particular,
\begin{equation}\label{eigenv-d-0}
\lim_{d\rightarrow 0^+} \lambda_1(d)= \max_{x\in\bar\Omega} \gamma(x).
\end{equation}
\end{thm}

\begin{proof}
Indeed, (\ref{eigenv-Wang}) is proved in \cite{LiCovilleWang} and we only explain how to derive (\ref{eigenv-d-0}).  Note that
$$
\lambda_1(d)
=\inf_{\phi\in C(\bar\Omega),\, \phi>0 \ \textrm{on} \ \bar\Omega} \sup_{x\in\Omega} {d \mathcal{K}[\phi]  + \gamma(x) \phi(x) \over \phi (x)}.
$$
Then it is easy to see that
$$
\lim_{d\rightarrow 0^+} \lambda_1(d)\leq  \lim_{d\rightarrow 0^+} \sup_{x\in\Omega}  \left( d \mathcal{K}[1]  + \gamma(x)  \right) =\max_{x\in\bar\Omega} \gamma(x).
$$
Moreover,
$$
\lim_{d\rightarrow 0^+} \lambda_1(d)\geq  \lim_{d\rightarrow 0^+}   \sup_{x\in\Omega}   \left( - d  + \gamma(x) \right)  =\max_{x\in\bar\Omega} \gamma(x).
$$
It is proved.
\end{proof}

Now we prepare a simple property, which will be repeatedly used in future.

\begin{lem}\label{lm-d=0}
Denote $\mu_0 = \lim_{d\rightarrow 0^+} \mu_{(\theta_d,0)}, \ \nu_D^0  = \lim_{d\rightarrow 0^+} \nu_{(0, \eta_D)}.$  Then $\mu_0 $ and $\nu_D^0$ are well defined. Moreover,
\begin{equation}\label{mu0nu0}
\mu_0= \sup \left\{\textrm{Re}\, \lambda\, |\, \lambda\in \sigma(D\mathcal{P}+ g(x, F_+(x,0), 0))   \right\},\ \nu_D^0 = \max_{x\in\bar\Omega}\  f(x,0,\eta_D).
\end{equation}
\end{lem}

This lemma follows directly from Proposition \ref{prop-d=0-theta} and  Theorem  \ref{thm-eigenv-general} and we omit the details of its proof.

\subsection{Properties of monotone systems}
The following result explains how to characterize the global dynamics of the competition model (\ref{original}) with two semi-trivial steady states.

\begin{thm}\label{thm-monotone}
Assume that the assumptions \textbf{(A0)}-\textbf{(A4)} hold and the system (\ref{original}) admits two semi-trivial steady states, denoted by $(\theta_d, 0)$ and $(0,\eta_D)$. We have the following three possibilities:
\begin{itemize}
\item[(i)]   If both $\mu_{(\theta_d,0)}>0$ and $\nu_{(0, \eta_D)}>0$,     the system (\ref{original}) at least has one positive steady state in $L^{\infty}(\Omega)\times L^{\infty}(\Omega)$.    If in addition, assume that the system (\ref{original}) has a unique positive steady state in $\mathbb X\times \mathbb X$, then it is globally asymptotically stable relative to
     $\mathbb X_{++} \times \mathbb X_{++}$.
\item[(ii)]   If $\mu_{(\theta_d,0)}>0$   and no positive steady states of the system (\ref{original}) exist, then the semi-trivial steady state $(0,\eta_D)$ is globally asymptotically stable relative to
     $\mathbb X_{++} \times \mathbb X_{++}$.
\item[(iii)]   If $\nu_{(0, \eta_D)}>0$  and the system (\ref{original}) admits  no positive steady states, then the semi-trivial steady state $(\theta_d,0)$ is globally asymptotically stable relative to
     $\mathbb X_{++} \times \mathbb X_{++}$.
\end{itemize}
\end{thm}

\begin{proof}
The arguments are almost the same as that of \cite[Theorem 2.1]{BaiLi2015}, where a simplified nonlocal operator is considered.
\end{proof}

The following simple property indicates that any positive steady state of the system (\ref{original}) in $L^{\infty}(\Omega)\times L^{\infty}(\Omega)$ belongs to $\mathbb X\times \mathbb X$ under the assumption \textbf{(A5)}.

\begin{lem}\label{lm-smoothness}
Assume that \textbf{(A0)} holds and $f_v g_u \leq f_u g_v$ in $\bar\Omega \times \mathbb R_+ \times \mathbb R_+$. Then any positive steady state of the system (\ref{original}) in $L^{\infty}(\Omega)\times L^{\infty}(\Omega)$ belongs to $\mathbb X\times \mathbb X$.
\end{lem}

This lemma can be verified easily by applying implicit function theorem and we omit the details.

\section{Limiting system as $d\rightarrow 0^+$}
In this section, we will study  the following  system
\begin{equation}\label{system d=0}
  \begin{cases}
    u= F_+(x,v),\\
    D\mathcal{P} [v]+v g(x,u,v)=0.
  \end{cases}
\end{equation}
and then demonstrate that it is the limiting system of the stationary problem of (\ref{original}) as $d \rightarrow 0^+$
\begin{equation}\label{stationary}
\begin{cases}
d \mathcal{K}[u]  +u f(x,u,v) =0  &\textrm{in } \Omega,\\
D \mathcal{P}[v]  +v g(x,u,v) =0  &\textrm{in } \Omega.
\end{cases}
\end{equation}

\subsection{Existence and uniqueness of limiting system}
The main purpose in this subsection is to establish the  following result concerning  the existence  and uniqueness of solution $(u,v)$ with $v$ being  positive to the  system (\ref{system d=0}).

\begin{thm}\label{thm-system-d=0}
Assume that \textbf{(A0)}-\textbf{(A5)} hold. The system (\ref{system d=0}) admits a unique solution $(u,v)$ with   $v$  being positive if and only if $\mu_0>0$. Moreover, $u\equiv 0$ if and only if $\nu_D^0 \le 0$. Here $\mu_0$ and $\nu_D^0$ are defined in (\ref{mu0nu0}).
\end{thm}

Notice that the system (\ref{system d=0}) can be rewritten as
\begin{equation}\label{system d=0-single}
    D\mathcal{P} [v]+v g(x,F_+(x,v),v )=0.
\end{equation}
First of all, we need establish a property of $g(x, F_+(x,v),v )$, which will be used repeatedly  throughout this paper.

\begin{lem}\label{lm-g-decreasing}
Assume that \textbf{(A2)}, \textbf{(A3)} and \textbf{(A5)} hold. Then   $  g(x,F_+(x,v),v )$ is strictly decreasing in $v\geq 0$. Moreover, for any fixed $M_1>0$, there exists $\sigma_1 = \sigma_1(M_1)>0$ such that for $0 \leq v< v_1\leq M_1$
$$
g(x,F_+(x,v_1),v_1 ) - g(x,F_+(x,v),v ) \leq - \sigma_1 (v_1-v).
$$
\end{lem}

\begin{proof}
Notice that due to \textbf{(A2)}, \textbf{(A3)} and (\ref{ift-F}), $F(x,v)$ is strictly  decreasing in $v\geq 0$.  Thus  we only need discuss the following three situations.
\begin{itemize}
\item In $\{ x\in \bar \Omega \ | \ F(x, v_1 ) \geq 0, F(x,v)>0\}$,
\begin{eqnarray*}
&&g(x,  F_+(x,  v_1 )  ,  v_1) - g(x, F_+(x,v),  v )\\
&=  & g(x,  F(x,  v_1 )   , v_1) - g(x,F(x,v), v)\\
&= &  (v_1-v) \left[   - g_u (x, F(x,\alpha ), \alpha) \frac{f_v (x, F(x,\alpha ), \alpha)}{ f_u (x, F(x,\alpha ), \alpha)}  + g_v (x, F(x,\alpha ), \alpha)\right]\\
&\leq &  -  c_1 (v_1-v),
\end{eqnarray*}
for some $ c_1>0$ due to \textbf{(A5)},  where
$0\leq v \leq \alpha   \leq  v_1\leq M_1.$

\item In $\{ x\in \bar\Omega \ | \ F(x, v_1) < 0, F(x,  v) \geq 0 \}$,
\begin{eqnarray*}
&&g(x,  F_+(x, v_1)    , v_1) - g(x,F_+(x,v),  v )\\
&=  & g(x,  0    ,  v_1) - g(x,F(x, v),  v )\\
&=&   g(x,  0    , v_1)  - g(x,  F(x, v_1)    , v_1)  +g(x,  F(x, v_1)    , v_1) - g(x,F(x,v), v )\\
&< &  (v_1-v)\left[   - g_u (x, F(x,\alpha ), \alpha) \frac{f_v (x, F(x,\alpha ), \alpha)}{ f_u (x, F(x,\alpha ), \alpha)}  + g_v (x, F(x,\alpha ), \alpha)\right]\\
&\leq &  -  c_2 (v_1-v) ,
\end{eqnarray*}
for some $c_2>0$ due to \textbf{(A5)},  where
$0\leq v \leq \alpha   \leq  v_1\leq M_1.$

\item In $\{ x\in\bar \Omega \ | \ F(x,  v_1) < 0, F(x,  v) < 0 \}$,
\begin{eqnarray*}
&&g(x,  F_+(x, v_1)    , v_1) - g(x,F_+(x, v),  v )\\
&=  & g(x,  0   , v_1) - g(x, 0 ,v ) =   (v_1-v) g_v (x, 0, \zeta)  \leq   -   c_3(v_1-v)
\end{eqnarray*}
for some $  c_3>0$ due to \textbf{(A5)}, where
$0\leq v \leq \zeta   \leq  v_1\leq M_1.$
\end{itemize}
Set $\sigma_1= \min \{ c_1, c_2, c_3 \}$ and the proof is complete.
\end{proof}

Next, we prepare the  comparison principle for (\ref{system d=0-single}).

\begin{lem}\label{cp-d=0}
Suppose that \textbf{(A0)}-\textbf{(A3)} and \textbf{(A5)} hold, $v^*, v\in C(\bar\Omega)$ are nonnegative, $v^* \not\equiv 0$, $v\not\equiv 0$ and $v^*, v$ satisfy the following inequalities respectively
$$D\mathcal{P} [v]+  v g(x,F_+(x,v),v )\ge 0\ \ \textrm{in}\  \bar\Omega,$$
$$D\mathcal{P} [v^*]+  v^* g(x,F_+(x,v^*),v^* )\le 0\ \ \textrm{in}\  \bar\Omega.$$
Then either $v^*> v$ or $v^*=v$ in $\bar\Omega$.
\end{lem}

\begin{proof}
First, we claim that {\it $v^* >0$ in $\bar\Omega$.}  Let $\Omega_0 = \{ x \in  \bar\Omega \ |\  v^* = 0 \}$.  Suppose that the claim is not true, then  $\Omega_0 \neq \emptyset $.   Obviously, $\Omega_0$ is closed in $\bar\Omega$.    Choose any $x_1\in \Omega_0$, then at $x=x_1$,
\begin{eqnarray*}
0 &\geq & D\mathcal{P} [v^*](x_1)+  v^*(x_1) g(x_1,F_+(x_1,v^*(x_1)),v^*(x_1) ) \\
  &=& D \int_{\Omega} p(x_1, y) v^*(y) dy.
\end{eqnarray*}
Due to \textbf{(A0)}, this  implies that $v^* = 0$ in a small neighborhood of $x_1$ in $\bar\Omega$. Thus $\Omega_0$ is open in $\bar\Omega$. Hence $\Omega_0 = \bar\Omega$, which is   a contradiction  to $v^* \not\equiv 0$. The claim is proved.

Now since $v^* >0$ in $\bar\Omega$, we have $\ell v^* >v$   in $\bar\Omega$ for $\ell$ large. Define
$$
\ell_* = \inf \left\{ \ell\ |\ \ell v^* >v \ \textrm{in}\ \bar\Omega \right\}.
$$
It is clear that $\ell_*>0$ since $v\not\equiv 0$.

We will further prove that $\ell_* \leq 1$.
Suppose that $\ell_*  > 1$. Let $z= \ell_* v^* -v$. It is clear that $z\geq 0$ and there exists $x_2 \in\bar\Omega$ such that $z(x_2 ) =0$. Thus by Lemma \ref{lm-g-decreasing}, direct computation yields that
\begin{eqnarray}\label{lem-cp-key}
D \mathcal{P} [z] &=&  \ell_* D\mathcal{P} [v^*] - D\mathcal{P} [v] \cr
                 &\leq & -\ell_* v^*g(x,F_+(x,v^*),v^* )  + v g(x,F_+(x,v),v ) \cr
                 & < &  -\ell_* v^*g(x,F_+(x, \ell_*v^*), \ell_*v^* )  + v g(x,F_+(x,\ell_*v^*),v ) \cr
                 &= & -  g(x,F_+(x, \ell_*v^*), \ell_*v^* )z  \cr
                  && -v \frac{g(x,F_+(x, \ell_*v^*), \ell_*v^* )  - g(x,F_+(x, \ell_*v^*), v) }{\ell_*v^* -v }   z.
\end{eqnarray}
Hence at $x=x_2$, we have
$$
D \int_{\Omega} p(x_2, y) z(y) dy < 0,
$$
which is impossible.  Therefore, $0< \ell_* \leq 1$.

At the end, if $\ell_* <1$, then $v^* > \ell_* v^*\geq v$ in $\bar\Omega$. If $\ell_* =1$, then similar to the arguments in the proof of the claim,  by using (\ref{lem-cp-key}), it follows that $z= v^*-v \equiv 0$ is the only possibility. The proof is complete.
\end{proof}

Now, the existence and uniqueness of  positive solutions to (\ref{system d=0-single}) will be characterized as follows.

\begin{prop}\label{prop-d=0-single}
Assume that \textbf{(A0)}-\textbf{(A5)} hold. The problem (\ref{system d=0-single}) admits a unique positive solution $V_0(x)$ in $C(\bar\Omega)$  if and only if $\mu_0>0$, where $\mu_0$ is defined in (\ref{mu0nu0}).
\end{prop}

\begin{proof}
On the one hand, assume that $V_0(x)$ is  a positive solution to (\ref{system d=0-single}). Then
\begin{eqnarray*}
 \mu_0 & =&\sup \left\{\textrm{Re}\, \lambda\, |\, \lambda\in \sigma(D\mathcal{P}+ g(x, F_+(x,0), 0))   \right\}\\
             &=&  \sup_{\phi\in C(\bar\Omega),\, \phi>0 \ \textrm{on} \ \bar\Omega} \inf_{x\in\Omega} {D \mathcal{P}[\phi]  + g(x, F_+(x,0), 0)\phi(x) \over \phi (x)}\\
             &\geq &  \inf_{x\in\Omega} {D \mathcal{P}[V_0 ]  + g(x, F_+(x,0), 0)V_0(x) \over V_0  (x)}\\
             & = & \inf_{x\in\Omega} \ \left[-g(x, F_+(x,V_0), V_0) +  g(x, F_+(x,0), 0) \right]\\
             & >& 0,
\end{eqnarray*}
since $g(x,F_+(x,v),v )$ is strictly decreasing in $v>0$  due to \textbf{(A2)}, \textbf{(A3)}, \textbf{(A5)} and (\ref{ift-F}).

On the other hand, assume that $\mu_0>0$. We first point out that $g(x,F_+(x, \alpha),v )$ is strictly increasing in $\alpha>0$ due to \textbf{(A2)}, \textbf{(A3)} and (\ref{ift-F}). This simple property will be used repeatedly for the rest of the proof.

By   \cite[Theorem 2.1]{BaiLi2017}, $\mu_0>0$ implies that
\begin{equation}\label{pf-v1}
D\mathcal{P} [v]+ v g(x,F_+(x,0),v )=0
\end{equation}
admits a unique positive solution, denoted by $V_1 \in \mathbb X$. For $k\geq 2$, let $V_k$ denotes the unique positive solution to
\begin{equation}\label{pf-vk}
D \mathcal{P}[v]  + v g(x,F_+(x,V_{k-1}),v ) =0.
\end{equation}
Note that due to \textbf{(A2)}, \textbf{(A3)} and (\ref{ift-F})
$$
\sup \left\{\textrm{Re}\, \lambda\, |\, \lambda\in \sigma(D\mathcal{P}+ g(x,F_+(x,V_{k-1}),0 ))   \right\}\geq \mu_0 >0.
$$
Thus the existence and uniqueness of $V_k$ is guaranteed by \cite[Theorem 2.1]{BaiLi2017}

Moreover, due to \textbf{(A2)}, \textbf{(A3)} and (\ref{ift-F}) again, it is routine to show that
$$
0<V_1< ... < V_k  < \eta_D,\ \ k\geq 1.
$$
Hence $V= \lim_{k\rightarrow \infty} V_k$ is a positive solution of (\ref{system d=0-single}) in $L^{\infty}(\Omega)$.
Moreover, it is standard to verify that $V\in C(\bar\Omega)$ due to \textbf{(A5)}.

At the end, suppose that  (\ref{system d=0-single}) has another positive solution $V^*\in C(\bar\Omega)$ and $V^* \not\equiv V$. First, based on the equations  (\ref{system d=0-single}) and (\ref{pf-v1}), one sees  that $V_1\leq V^* $ due to \textbf{(A2)}, \textbf{(A3)} and (\ref{ift-F}).  Then by  the equations  (\ref{system d=0-single}) and (\ref{pf-vk}), \textbf{(A2)}, \textbf{(A3)} and (\ref{ift-F}) yield that $V_k\leq V^*$, $k=2,3,...$, which implies that $V\leq V^*.$ Thanks to Lemma \ref{cp-d=0}, we have $V< V^*$ in $\bar\Omega$.

Define
$$
\ell_1 = \inf \left\{ \ell\ |\ \ell V >V^* \ \textrm{in}\ \bar\Omega \right\}.
$$
Obviously, $\ell _1>1$ since $V< V^*$ in $\bar\Omega$. Then  thanks to Lemma \ref{lm-g-decreasing},
\begin{eqnarray}\label{pf-unique}
D\mathcal{P} [\ell_1 V]+  \ell_1 V   g(x,F_+(x,\ell_1 V), \ell_1 V )
                 <   D\mathcal{P} [\ell_1 V]+  \ell_1 V g(x,F_+(x, V),   V )=0.
\end{eqnarray}
Again, by Lemma \ref{cp-d=0}, we have either $\ell_1 V\equiv V^*$ or $\ell_1 V > V^*$ in $\bar\Omega$. The former case contradicts to (\ref{pf-unique}), while the latter case is a contradiction to the definition of $\ell_1$. Therefore, the positive solution of (\ref{system d=0-single}) is unique. The proof is complete.
\end{proof}

Finally, we are ready to demonstrate the main result in this subsection.

\begin{proof}[Proof of Theorem \ref{thm-system-d=0}]
Notice that if  $V_0$ is a positive solution of (\ref{system d=0-single}), then $(u,v)= (F_+(x,V_0),V_0)$ is a solution of the limiting system (\ref{system d=0}). Thus it is clear that the first part  is equivalent to Proposition \ref{prop-d=0-single}.

It remains to verify that $u\equiv 0$ if and only if $\nu_D^0 \le 0$. This is obvious since $u\equiv 0$ if and only if $v=\eta_D$.
\end{proof}

\subsection{Relations between systems (\ref{system d=0}) and (\ref{stationary})}

In this subsection, we will demonstrate that (\ref{system d=0}) is the limiting system of (\ref{stationary}) as $d\rightarrow 0^+$ by characterizing  the asymptotic behavior of the solutions of (\ref{stationary}) as $d\rightarrow 0^+$.

\begin{thm}\label{thm-asymptotic}
Suppose that \textbf{(A0)}-\textbf{(A5)} hold and $(u_d, v_d)$ is a positive solution of the system (\ref{stationary}).  If $\mu_0>0$, there exist $C_2, C_3 >0$ such that
$$
\left( U_0 -C_2 d \right)_+ < u_d < U_0 + C_2 \sqrt{d}, \  ( V_0 - C_3 \sqrt{d}  )_+ < v_d < V_0 + C_3 d\ \ \textrm{in} \ \bar\Omega,
$$
where $(U_0, V_0)$  is a solution to (\ref{system d=0}), while the existence and uniqueness of  $V_0$   is  guaranteed by Proposition \ref{prop-d=0-single}.
\end{thm}

\begin{proof}
For clarity, we divide the proof  into several steps. Note that $0< u_d, v_d < M$ due to \textbf{(A4)}.
Then due to  \textbf{(A2)}, there exists $c_0>0 $ such that
$$
f_u(x,u,v) \leq -c_0\ \ \textrm{for}\  x\in \bar\Omega, u,v \in [0,M+1].
$$

{\it Step 1.} Regard $v_d$ as  a given function first and  thus $u_d$ is unique due to \textbf{(A2)}. Set $\hat u = F_+(x, v_d) + c_1 \sqrt{d}$ with $c_1>0$. Then by direct computation, we have
\begin{eqnarray*}
&&  d\mathcal{K}[\hat u] + \hat u f(x, \hat u, v_d)\\
& \leq & d\int_{\Omega} k(x,y ) \left (F_+(y, v_d(y)) + c_1 \sqrt{d}\right) dy \\
&&  + \left ( F_+(x, v_d) + c_1 \sqrt{d}  \right)f(x, F_+(x, v_d) + c_1 \sqrt{d} , v_d)\\
& \leq & d\int_{\Omega} k(x,y ) \left (F_+(y, v_d(y)) + c_1 \sqrt{d}\right) dy \\
&&  + \left ( F_+(x, v_d) + c_1 \sqrt{d}  \right)\left[ f(x, F_+(x, v_d) + c_1 \sqrt{d} , v_d) - f(x, F_+(x, v_d) , v_d)\right]  \\
&\leq & d\int_{\Omega} k(x,y ) \left (F_+(y, v_d(y)) + c_1 \sqrt{d}\right) dy - c_0c_1 \sqrt{d} \left ( F_+(x, v_d) + c_1 \sqrt{d}  \right)\\
&< & 0,
\end{eqnarray*}
if $c_1$ is large enough and fix it.

Set $\underline u =  (F_+(x, v_d) -c_2 d)_+$ with $c_2>0$. Then it follows that
\begin{eqnarray*}
&&  d\mathcal{K}[\underline u] + \underline u f(x, \underline u, v_d)\\
& \geq   & -d (F_+(x, v_d) -c_2 d)_+ + (F_+(x, v_d) -c_2 d)_+  f(x, (F_+(x, v_d) -c_2 d)_+, v_d)\\
&=& -d (F_+(x, v_d) -c_2 d)_+  \\
&& +  (F_+(x, v_d) -c_2 d)_+ \left[f(x, F_+(x, v_d) -c_2 d, v_d) - f(x, F_+(x, v_d) , v_d)  \right] \\
&\geq & (F_+(x, v_d) -c_2 d)_+  \left( -d  + c_0 c_2 d \right)\\
&\geq & 0
\end{eqnarray*}
if $c_2$ is large enough and fix it.

Now  we have derived that
\begin{equation}\label{pf-thm-i-ud}
(F_+(x, v_d) -c_2 d)_+ < u_d < F_+(x, v_d) + c_1 \sqrt{d} \ \ \textrm{in} \ \bar\Omega
\end{equation}
by upper/lower solution method.

{\it Step 2.}  Let us give a preliminary estimate of $v_d$. Since $\mu_0>0$, then similar to the proof  of  Proposition \ref{prop-d=0-single}, for $d$ sufficiently small, the following problem
\begin{equation}\label{pf-thm-hat v}
D\mathcal{P}[v] + v g(x, (F_+(x, v) -c_2 d)_+ , v) =0
\end{equation}
admits a unique positive solution, denoted by $\hat v$. Similarly, for $d$ sufficiently small, the problem
\begin{equation}\label{pf-thm-underline v}
D\mathcal{P}[v] + v g(x,  F_+(x, v) + c_1 \sqrt{d} , v) =0
\end{equation}
admits a unique positive solution, denoted by $\underline v$. According to (\ref{pf-thm-i-ud}), it is standard to check that $v_d$ is a lower solution of (\ref{pf-thm-hat v}). Together with the fact that $\hat v$ is the unique solution of (\ref{pf-thm-hat v}), one sees that $v_d < \hat v$  in  $\bar\Omega.$ Similarly, we have $\underline v < v_d$  in  $\bar\Omega.$ Hence
\begin{equation}\label{pf-thm-preliminary}
\underline v <  v_d < \hat v \ \ \textrm{in}\ \bar\Omega.
\end{equation}

{\it Step 3.} Now we could improve the estimate of $v_d$ and relate it to $V_0$. On the one side, consider $(1+c_3d)V_0$ with $c_3> 0$ and compute as follows
\begin{eqnarray}\label{pf-thm>hat v}
&& D\mathcal{P} [(1+c_3d)V_0 ] + (1+c_3d)V_0 g(x,  (F_+(x, (1+c_3d)V_0) -c_2 d)_+ , (1+c_3d)V_0)\cr
&=&   (1+c_3d) V_0 \left[ - g\left(x,F_+(x,V_0),V_0 \right)+ g \left(x,  (F_+(x, (1+c_3d)V_0) -c_2 d)_+ , (1+c_3d)V_0 \right)\right]\cr
&\leq & (1+c_3d) V_0 \left[ - g(x,F_+(x,V_0),V_0 )+ g(x,  F_+(x, (1+c_3d)V_0) -c_2 d  , (1+c_3d)V_0)\right]\cr
& =  &(1+c_3d) V_0 \left[ - g(x,F_+(x,V_0),V_0 )+ g(x,  F_+(x, (1+c_3d)V_0)    , (1+c_3d)V_0)\right]\cr
&&   -c_2 d (1+c_3d) V_0   g_u(x, \xi , (1+c_3d)V_0)\cr
&\leq & -d (1+c_3d) V_0 \left[ \sigma_1 c_3   V_0    + c_2   g_u(x, \xi , (1+c_3d)V_0) \right] \cr
&< & 0,
\end{eqnarray}
if $c_3$ is large enough and fix it,
where
$$
F_+(x, (1+c_3d)V_0) -c_2 d \leq \xi= \xi(x)\leq F_+(x, (1+c_3d)V_0),
$$
the second inequality is due to Lemma \ref{lm-g-decreasing} and $\sigma_1>0$ is some constant.

This indicates that $(1+c_3d)V_0$ is a upper solution of (\ref{pf-thm-hat v}). Since $\hat v$ is the unique solution of (\ref{pf-thm-hat v}), we have
\begin{equation}\label{pf-thm-further-hat v}
\hat v < (1+c_3d)V_0\ \ \textrm{in}\ \bar\Omega.
\end{equation}

On the other side, consider $(1- c_4 \sqrt{d}) V_0$ with $c_4>0$ and compute as follows
\begin{eqnarray*}
&& D\mathcal{P}[(1- c_4 \sqrt{d}) V_0] + (1- c_4 \sqrt{d}) V_0 g (x,  F_+(x, (1- c_4 \sqrt{d}) V_0) + c_1 \sqrt{d} , (1- c_4 \sqrt{d}) V_0 ) \\
&= & (1- c_4 \sqrt{d}) V_0 \left[  -g\left(x,F_+(x,V_0),V_0 \right) +   g (x,  F_+(x, (1- c_4 \sqrt{d}) V_0) + c_1 \sqrt{d} , (1- c_4 \sqrt{d}) V_0 )   \right]\\
&= & (1- c_4 \sqrt{d}) V_0 \left[  -g\left(x,F_+(x,V_0),V_0 \right) +   g (x,  F_+(x, (1- c_4 \sqrt{d}) V_0)  , (1- c_4 \sqrt{d}) V_0 )   \right]\\
&&  + c_1 \sqrt{d}(1- c_4 \sqrt{d}) V_0  g_u (x,  \xi , (1- c_4 \sqrt{d}) V_0 )\\
&\geq &  \sqrt{d}(1- c_4 \sqrt{d}) V_0 \left[\sigma_2 c_4   V_0 + c_1    g_u (x,  \xi , (1- c_4 \sqrt{d}) V_0 )\right]\\
&>&  0,
\end{eqnarray*}
if $c_4$ is large enough and fix it,
where
$$
(1- c_4 \sqrt{d}) V_0  \leq \xi= \xi(x)\leq (1- c_4 \sqrt{d}) V_0 + c_1 \sqrt{d},
$$
the  inequality is due to Lemma \ref{lm-g-decreasing} and $\sigma_2>0$ is some constant.

Similarly, this implies that $(1- c_4 \sqrt{d}) V_0$ is a lower solution of (\ref{pf-thm-underline v}). Thus the uniqueness of the positive solution to (\ref{pf-thm-underline v}) yields that
\begin{equation}\label{pf-thm-further-underline v}
(1- c_4 \sqrt{d}) V_0 <  \underline v  \ \ \textrm{in}\ \bar\Omega.
\end{equation}

At the end, it follows from (\ref{pf-thm-preliminary}), (\ref{pf-thm-further-hat v}) and (\ref{pf-thm-further-underline v})   that
\begin{equation}\label{pf-thm-final}
(1- c_4 \sqrt{d}) V_0 < \underline v <  v_d < \hat v < (1+c_3d)V_0 \ \ \textrm{in}\ \bar\Omega.
\end{equation}
This, together with (\ref{pf-thm-i-ud}), gives the desired conclusion.
\end{proof}

\section{Proof of Theorem \ref{thm-main}}

For clarity, we prove Theorem \ref{thm-main}(ii) and (iii) first since it is easier.

\begin{proof}[Proof of Theorem \ref{thm-main}(ii)]
The assumptions in Theorem \ref{thm-main}(ii), Lemma \ref{lm-d=0} and Theorem \ref{thm-monotone} together indicate that to prove   Theorem \ref{thm-main}(ii), it suffices to show that the system (\ref{original}) admits no positive steady states if $d$ is sufficiently small.

Suppose that   there exists a sequence  $\{ d_i\}_{i\geq 1}$ with $\lim_{i\rightarrow \infty} d_i=0^+$ such that for $d=d_i$,  the problem as follows admits a positive solution $(u_i, v_i)\in \mathbb X \times \mathbb X$
\begin{equation}\label{thm-pf-i eq}
\begin{cases}
d_i \mathcal{K}[u]  +u f(x,u,v)=0 &\textrm{in } \Omega,\\
D \mathcal{P}[v]  +v g(x,u,v)=0 &\textrm{in } \Omega.
\end{cases}
\end{equation}

Thanks to Theorem \ref{thm-asymptotic}, we have
$$
\lim_{i\rightarrow \infty} u_i = U_0 = F_+(x, V_0),\ \lim_{i\rightarrow \infty} v_i = V_0 \  \ \textrm{in} \ \mathbb X.
$$
Note that by Theorem \ref{thm-system-d=0}, the positivity of $V_0$ is guaranteed by $\mu_0>0$, while $\nu_D^0< 0$ implies that $U_0 \equiv 0$. Thus, in fact, $V_0= \eta_D$. Then due to the  first  equation in (\ref{thm-pf-i eq}), one sees that
\begin{eqnarray*}
0  =  \inf_{\phi\in C(\bar\Omega),\, \phi>0 \ \textrm{on} \ \bar\Omega} \sup_{x\in\Omega} {d_i \mathcal{K}[\phi]  + f(x, u_i, v_i) \phi \over \phi }\leq   \sup_{x\in\Omega}  \left\{ d_i \mathcal{K}[1]  + f(x, u_i, v_i)\right\}.
\end{eqnarray*}
Therefore,
$$
0\leq \lim_{i\rightarrow \infty}\sup_{x\in\Omega}  \left\{ d_i \mathcal{K}[1]  + f(x, u_i, v_i)\right\} = \max_{x\in\bar\Omega}\  f(x,0,\eta_D) = \nu_D^0< 0.
$$
This is a contradiction.

\end{proof}

\begin{proof}[Proof of Theorem \ref{thm-main}(iii)]
First of all, we claim that {\it if $\mu_0<0$, then $\nu_D^0 >0$.}

Recall that
\begin{eqnarray*}
0>\mu_0 &= & \sup \left\{\textrm{Re}\, \lambda\, |\, \lambda\in \sigma(D\mathcal{P}+ g(x, F_+(x,0), 0))   \right\}\\
 &=&  \sup_{\phi\in C(\bar\Omega),\, \phi>0 \ \textrm{on} \ \bar\Omega} \inf_{x\in\Omega} {D \mathcal{P}[\phi]  + g(x, F_+(x,0), 0) \phi \over \phi }\\
 &\geq & \inf_{x\in\Omega} {D \mathcal{P}[\eta_D]  + g(x, F_+(x,0), 0 ) \eta_D \over  \eta_D }\\
 & = &  \inf_{x\in\Omega}  \left[ -g(x, 0 ,\eta_D)    +  g(x, F_+(x,0), 0)\right]\\
 & > &\inf_{x\in\Omega}  \left[ -g(x, 0 ,\eta_D)    +  g(x, F_+(x,\eta_D), \eta_D)\right],
\end{eqnarray*}
where the last inequality is  due to Lemma \ref{lm-g-decreasing}. This indicates that there exists $x^*\in\bar\Omega$ such that
$$
F_+(x^*,\eta_D(x^*)) = F(x^*,\eta_D(x^*))>0.
$$
Hence according to \textbf{(A2)} and the definition of $F$, it follows that
$$
\nu_D^0 = \max_{x\in\bar\Omega}\  f(x,0,\eta_D) \geq     f(x^*,0,\eta_D(x^*)) >0.
$$
The claim is proved.

Now by  Lemma \ref{lm-d=0} and Theorem \ref{thm-monotone},  to prove   Theorem \ref{thm-main}(iii), it suffices to show that the system (\ref{original}) admits no positive steady states if $d$ is sufficiently small.

Suppose that   there exists a sequence  $\{ d_i\}_{i\geq 1}$ with $\lim_{i\rightarrow \infty} d_i=0^+$ such that for $d=d_i$,  the problem (\ref{thm-pf-i eq})  admits a positive solution $(u_i, v_i)\in \mathbb X \times \mathbb X$.  Then
\begin{eqnarray}\label{thm-pf-(iii)}
0>\mu_0 &= & \sup \left\{\textrm{Re}\, \lambda\, |\, \lambda\in \sigma(D\mathcal{P}+ g(x, F_+(x,0), 0))   \right\}\cr
 &=&  \sup_{\phi\in C(\bar\Omega),\, \phi>0 \ \textrm{on} \ \bar\Omega} \inf_{x\in\Omega} {D \mathcal{P}[\phi]  + g(x, F_+(x,0), 0) \phi \over \phi }\cr
 &\geq & \inf_{x\in\Omega} {D \mathcal{P}[v_i]  + g(x, F_+(x,0), 0 ) v_i \over  v_i }\cr
 & = &  \inf_{x\in\Omega}  \left[ -g(x, u_i , v_i)    +  g(x, F_+(x,0), 0)\right]\cr
 &>  & \inf_{x\in\Omega}  \left[ -g(x, u_i , v_i)    + g(x, F_+(x, v_i), v_i )\right],
\end{eqnarray}
where the last inequality is  due to Lemma \ref{lm-g-decreasing}.

Moreover, notice that regardless of whether $\mu_0>0$ or not,  the estimate in (\ref{pf-thm-i-ud})  in the proof of Theorem \ref{thm-asymptotic} always hold.
Hence it follows  from (\ref{thm-pf-(iii)}) that
\begin{eqnarray*}
0>\mu_0  &\geq & \inf_{x\in\Omega}  \left[ -g(x, u_i , v_i)    + g(x, F_+(x, v_i), v_i )\right] \\
& \geq  & \inf_{x\in\Omega}   \left[ -g(x,(F_+(x, v_i) -c_2 d_i)_+,v_i)    +  g(x, F_+(x,v_i), v_i) \right],
\end{eqnarray*}
which yields a contradiction by letting $d_i \rightarrow 0^+$.
\end{proof}

The rest of this section is devoted to the  proof of Theorem \ref{thm-main}(i).

\begin{proof}[Proof of Theorem \ref{thm-main}(i)]
According to   Lemma \ref{lm-d=0},  Lemma \ref{lm-smoothness} and Theorem \ref{thm-monotone}, it is clear that we only need   verify the uniqueness of positive steady states in $\mathbb X\times \mathbb X$  when $d$ is sufficiently small.

Suppose that there exists a sequence
$\{d_i \}_{i\geq 1}$ with $\lim_{i\rightarrow \infty  }d_i =0^+$  such that the system (\ref{original}) with $d=d_i$ admits two different positive steady states $(u_i,v_i)$ and $(u_i^*,v_i^*)$ in   $\mathbb X\times \mathbb X$.   Then similar to the discussion at the beginning of the proof of \cite[Theorem 1.1(i)]{BaiLi2017}, we assume that  w.l.o.g., $u_i<u_i^*$, $v_i>v_i^*$.

Denote
$$
\phi_i=\frac{u_i-u^*_i}{\|u_i-u^*_i\|_{L^2(\Omega)}+\|v_i-v^*_i\|_{L^2(\Omega)}}<0,\ \psi_i=\frac{v_i-v^*_i}{\|u_i-u^*_i\|_{L^2(\Omega)}+\|v_i-v^*_i\|_{L^2(\Omega)}}>0.
$$
It is routine to check that
\begin{equation}\label{pf-thm-phipsi}
\begin{cases}
\displaystyle d_i\mathcal{K}[\phi_i] + f(x, u_i, v_i)\phi_i  \\
 \ \ \ \  \displaystyle = -u^*_i\left[ \frac{f(x,u_i, v_i)- f(x, u_i^*, v_i)}{u_i- u_i^*}\phi_i + \frac{f(x,u_i^*, v_i)- f(x, u_i^*, v_i^*)}{v_i- v_i^*}\psi_i \right] \\
 \ \ \ \  \displaystyle  =  -u^*_i f_u (x, \xi_i, v_i )\phi_i  -u^*_i f_v (x, u_i^*, \zeta_i )\psi_i  & {\rm in~} \Omega\\
\displaystyle D \mathcal{P}[{\psi}_i]+ g(x,u_i, v_i)\psi_i  \\
\ \ \ \ \displaystyle = - v_i^* \left[  \frac{g(x,u_i, v_i)- g(x, u_i^*, v_i)}{u_i- u_i^*}\phi_i + \frac{g(x,u_i^*, v_i)- g(x, u_i^*, v_i^*)}{v_i- v_i^*}\psi_i \right] \\
 \ \ \ \  \displaystyle  =  -v^*_i g_u (x, \alpha_i, v_i )\phi_i  -v^*_i g_v (x, u_i^*, \beta_i )\psi_i  & {\rm in~} \Omega,
\end{cases}
\end{equation}
where
$$
u_i \leq \xi_i, \alpha_i \leq  u_i^*,\ v_i^* \leq \zeta_i, \beta_i \leq v_i.
$$

For simplicity, denote
$$
k* \phi = \int_{\Omega} k(x,y) \phi(y) dy.
$$
Direct calculation gives that
\begin{eqnarray}\label{thm-pf-c}
0> \phi_i &=& \frac{ -f_v (x, u_i^*, \zeta_i ) u^*_i \psi_i-d_i k* \phi_i}{f(x,u_i ,v_i) + f_u (x, \xi_i, v_i ) u_i^* -d_i } \cdot{u_i\over u_i} \cr
    &=&  \frac{ -f_v (x, u_i^*, \zeta_i ) u_iu^*_i \psi_i-d_i u_i k* \phi_i}{- d_ik*u_i + f_u (x, \xi_i, v_i )  u_iu_i^*  }\cr
  &\geq &  - \frac{f_v (x, u_i^*, \zeta_i )}{f_u (x, \xi_i, v_i )}\psi_i +   \left( {k*u_i\over u_i}  -  f_u (x, \xi_i, v_i ) {u_i^*\over d_i} \right)^{-1}  k* \phi_i.
\end{eqnarray}
On the other hand, notice that
\begin{eqnarray*}
0 & =& \sup \left\{\textrm{Re}\, \lambda\, |\, \lambda\in \sigma(D\mathcal{P}+ g(x, u_i, v_i))   \right\} \\
      &=&   \sup_{\phi\in C(\bar\Omega),\, \phi>0 \ \textrm{on} \ \bar\Omega} \inf_{x\in\Omega}   {D \mathcal{P}[\phi]  + g(x, u_i, v_i) \phi   \over \phi }  \\
      &\geq & \inf_{x\in\Omega}   {D \mathcal{P}[\psi_i]  + g(x, u_i, v_i) \psi_i   \over \psi_i }\\
      &=& \inf_{x\in\Omega}  -v_i^* \left[ g_u (x, \alpha_i, v_i ){\phi_i\over \psi_i}  +  g_v (x, u_i^*, \beta_i ) \right].
\end{eqnarray*}
This, together with (\ref{thm-pf-c}), yields that
\begin{eqnarray}\label{pf-thm-key}
0 & \geq &\inf_{x\in\Omega}  \left[  - g_u (x, \alpha_i, v_i )\phi_i -  g_v (x, u_i^*, \beta_i )\psi_i  \right]\cr
&\geq &  \inf_{x\in\Omega}     \left( g_u (x, \alpha_i, v_i ) \frac{f_v (x, u_i^*, \zeta_i )}{f_u (x, \xi_i, v_i )} -  g_v (x, u_i^*, \beta_i )\right) \psi_i \cr
  &&\ \ \ \ \  - g_u (x, \alpha_i, v_i )   \left( {k*u_i\over u_i}  -  f_u (x, \xi_i, v_i ) {u_i^*\over d_i} \right)^{-1}  k* \phi_i.
\end{eqnarray}

We claim that
\begin{equation}\label{pf-thm-obstacle}
\lim_{i\rightarrow \infty } \left( {k*u_i\over u_i}  -  f_u (x, \xi_i, v_i ) {u_i^*\over d_i} \right)^{-1} =0 \ \ \textrm{in} \ \mathbb X.
\end{equation}
Since $\mu_0>0$, by Theorem \ref{thm-asymptotic}, one has
$$
\lim_{i\rightarrow \infty} u_i = U_0 = F_+(x, V_0)  \  \ \textrm{in} \ \mathbb X
$$
and $U_0 \not\equiv 0$ because of Theorem \ref{thm-system-d=0}. Then the extra assumption that $k(x,y)>0$ for $x, y\in \bar\Omega$ indicates that $k* U_0 >0$ in $\bar\Omega$. Hence it follows that
\begin{eqnarray*}
&& \lim_{i\rightarrow \infty } \left( {k*u_i\over u_i}  -  f_u (x, \xi_i, v_i ) {u_i^*\over d_i} \right)^{-1} \leq  \lim_{i\rightarrow \infty } \left( {k*u_i\over u_i}  -  f_u (x, \xi_i, v_i ) {u_i\over d_i} \right)^{-1} \\
&\leq &  \lim_{i\rightarrow \infty } \left(  -  f_u (x, \xi_i, v_i )  {k*u_i\over   d_i} \right)^{-1/2}=0.
\end{eqnarray*}
The claim is proved.

Next, by Theorem \ref{thm-asymptotic} and the assumption \textbf{(A5)}, it is easy to check that
\begin{eqnarray}\label{pf-thm-A5}
&& \lim_{i\rightarrow \infty} \left( g_u (x, \alpha_i, v_i ) \frac{f_v (x, u_i^*, \zeta_i )}{f_u (x, \xi_i, v_i )} -  g_v (x, u_i^*, \beta_i)\right) \cr
&=& g_u (x, U_0, V_0 ) \frac{f_v (x, U_0, V_0  )}{f_u (x, U_0, V_0  )} -  g_v (x, U_0, V_0 ) \geq c_0 >0\ \ \textrm{in} \ \mathbb X,
\end{eqnarray}
where $c_0$ is some constant.

Then (\ref{pf-thm-key}), (\ref{pf-thm-obstacle}) and (\ref{pf-thm-A5}) together imply that for $i$ large enough,   there exists $x_i \in \bar\Omega$ such that
\begin{equation}\label{pf-thm-psi-0}
\lim_{i\rightarrow \infty} \psi_i(x_i) =0.
\end{equation}
This, combined with (\ref{thm-pf-c}) and (\ref{pf-thm-obstacle}), gives that
\begin{equation}\label{pf-thm-phi-0}
\lim_{i\rightarrow \infty} \phi_i(x_i)=0.
\end{equation}

Back to the equation satisfied by $\psi_i$ in  (\ref{pf-thm-phipsi}), it follows from (\ref{pf-thm-psi-0}) and (\ref{pf-thm-phi-0})  that
\begin{equation}\label{pf-thm-p*psi-xi-0}
\lim_{i\rightarrow \infty}  D (p*\psi_i)(x_i) =  0.
\end{equation}
We claim that  {\it there exists a sequence of $\{ \psi_i \}_{i\geq 1} $, still denoted by $\{ \psi_i \}_{i\geq 1} $, such that}
\begin{equation}\label{pf-thm-p*psi-0}
\lim_{i\rightarrow \infty}  p*\psi_i =  0 \ \ \textrm{in}\ \mathbb X.
\end{equation}

Obviously, by passing to a subsequence if necessary, there exist $x_0 \in \bar\Omega$ and $\Psi_0\in \mathbb X$ such that
$$
\lim_{i\rightarrow \infty} x_i = x_0\ \ \textrm{and}\ \ \lim_{i\rightarrow \infty}  p*\psi_i =  \Psi_0 \ \ \textrm{in}\ \mathbb X.
$$
Thus $\Psi_0 (x_0) =0$. Denote $\Omega_0 = \{x\in \bar\Omega \ |\ \Psi_0 (x) =0 \} $. To prove the claim, it suffices to show that $\Omega_0$ is both open and closed in $\bar\Omega$.  According to \textbf{(A0)}, there exist $c>0$ and $\delta>0$ such that
$$
p(x,y) \geq c>0  \ \ \textrm{if}\ x, y \in\bar\Omega, |x-y|\leq \delta.
$$
Hence
$$
0 = \Psi_0 (x_0) = \lim_{i\rightarrow \infty}  (p*\psi_i )(x_0)\geq c\lim_{i\rightarrow \infty} \int_{\{x\in \bar\Omega, |x-x_0|\leq \delta \}}\psi_i(x)d x \geq 0,
$$
which yields that
$$
\lim_{i\rightarrow \infty} \psi_i =0\ \ a.e. \ \textrm{in}\ \{x\in \bar\Omega, |x-x_0|\leq \delta \}.
$$
Moreover, based on the equations satisfied by $v_i$ and $\psi_i$ respectively, it is routine to check that
\begin{eqnarray*}
 p*\psi_i &=&  \left( {p*v_i \over v_i} - {g_v(x, u_i^*, \beta_i ) \over D} v_i^*\right) \psi_i -  {g_u(x, \alpha_i, v_i )\over D}v_i^*\phi_i \\
  &\leq&  \left( {p*v_i \over v_i} - {g_v(x, u_i^*, \beta_i ) \over D} v_i^*\right) \psi_i,
\end{eqnarray*}
which yields that
$$
\lim_{i\rightarrow \infty} p* \psi_i =0\ \ a.e. \ \textrm{in}\ \{x\in \bar\Omega, |x-x_0|\leq \delta \}.
$$
Hence $\Psi_0 = 0$ in $\{x\in \bar\Omega, |x-x_0|< \delta \}$ and thus $\Omega_0$ is open. On the other side, the closedness  of  $\Omega_0$  is obvious. (\ref{pf-thm-p*psi-0}) is verified.

At the  end,  based on the equations satisfied by $v_i$ and $\psi_i$ again, we derive that
\begin{eqnarray*}
&& Dv_i p* \psi_i = D\psi_i p* v_i    -v_iv^*_i g_u (x, \alpha_i, v_i )\phi_i  -v_iv^*_i g_v (x, u_i^*, \beta_i )\psi_i\\
&\geq &  \left[D  p* v_i + v_iv^*_i\left( g_u (x, \alpha_i, v_i )\frac{f_v (x, u_i^*, \zeta_i )}{f_u (x, \xi_i, v_i )} - g_v (x, u_i^*, \beta_i )\right) \right] \psi_i\\
&&  -v_iv^*_i g_u (x, \alpha_i, v_i )  \left( {k*u_i\over u_i}  -  f_u (x, \xi_i, v_i ) {u_i^*\over d_i} \right)^{-1}  k* \phi_i.
\end{eqnarray*}
Then thanks to \textbf{(A5)}, Theorem \ref{thm-asymptotic}, (\ref{pf-thm-obstacle}) and (\ref{pf-thm-p*psi-0}), one immediately sees that
$$
\lim_{i\rightarrow \infty} \psi_i =0\ \ \ \textrm{in}\ L^{\infty}(\Omega),
$$
which, combined with (\ref{thm-pf-c}) and (\ref{pf-thm-obstacle}), implies that
$$
\lim_{i\rightarrow \infty} \phi_i =0\ \ \ \textrm{in}\ L^{\infty}(\Omega).
$$
This is a contradiction to the definitions of $\phi_i$ and $\psi_i$.
\end{proof}


\begin{thebibliography}{99999}






\bibitem{BaiLi2015}
X. Bai and F. Li,
\newblock  Global dynamics of a competition model with nonlocal dispersal II: The full system,
\newblock\emph{J. Differential Equations} \textbf {258} (2015), 2655--2685.


\bibitem{BaiLi2017}
X. Bai and F. Li,
\newblock Classification of global dynamics of competition models with nonlocal
 dispersals I: Symmetric kernels,
\newblock preprint.



\bibitem{CC98}
R. S. Cantrell and C. Cosner,
\newblock  On the effects of spatial heterogeneity on the persistence of interacting species,
\newblock\emph{J. Math. Biol.} \textbf{37} (1998), 103--145.


\bibitem{CC-book}
R.S. Cantrell and C. Cosner,
\newblock\emph{Spatial ecology via reaction-diffusion equations,}
\newblock Wiley Series in Mathematical and Computational Biology. John Wiley and Sons, Ltd., Chichester, 2003.








\bibitem{Fife}
P.C. Fife,
\newblock An integrodifferential analog of semilinear parabolic PDEs, Partial differential equations
and applications,
\newblock\emph{Lecture Notes in Pure and Appl. Math., Dekker, New York,} \textbf{177} (1996), 137--145.



\bibitem{Grinfeld}
M. Grinfeld, G. Hines, V. Hutson, K. Mischaikow and G.T. Vickers,
\newblock  Non-local dispersal,
\newblock\emph{Differ Integral Equ} \textbf{18} (2005), no. (11), 1299--1320.


\bibitem{HeNi}
X. He and W.-M. Ni,
\newblock Global dynamics of the Lotka-Volterra competition-diffusion system: Diffusion and spatial heterogeneity, I,
\newblock\emph{Communication on Pure and Applied Mathematics} \textbf{69} (2016), 981--1014.



\bibitem{HNShen2012}
G. Hetzer; T. Nguyen and W. Shen,
\newblock  Coexistence and extinction in the Volterra-Lotka competition model with nonlocal dispersal,
\newblock\emph{Commun. Pure Appl. Anal.}  \textbf{11}  (2012), 1699--1722.



\bibitem{HMMV}
V. Huston, S. Martinez, K. Miscaikow and G.T. Vichers,
\newblock   The evolution of dispersal,
\newblock \emph{J. Math. Biol.} \textbf{47} (2003), 483--517.




\bibitem{LamNi}
K.-Y. Lam and W.-M. Ni,
\newblock Uniqueness and complete dynamics in the heterogeneous competition-diffusion systems,
\newblock\emph{SIAM J. Appl. Math.}   \textbf{72} (2012), 1695--1712.




\bibitem{LiCovilleWang}
F. Li,  J. Coville and X. Wang,
\newblock On eigenvalue problems arising from nonlocal diffusion models,
\newblock \emph{Discrete Contin. Dyn. Syst.}   {\bf 37} (2017),  879--903.


\bibitem{LLW14}
F. Li, Y. Lou and Y. Wang,
\newblock   Global dynamics of a competition model with non-local dispersal I: the shadow system,
\newblock \emph{J. Math. Anal. Appl.}   {\bf 412} (2014), 485--497.

\bibitem{LWW11}
F. Li, L. Wang, and Y. Wang,
\newblock On the effects of migration and inter-specific competitions in steady state of some Lotka-Volterra model,
\newblock \emph{Discrete Contin. Dyn. Syst. Ser. B}    {\bf 15} (2011), 669--686.




\bibitem{Lou}
Y. Lou,
\newblock On the effects of migration and spatial heterogeneity on single and multiple species,
\newblock \emph{J. Differential Equations} \textbf{223} (2006), 400--426.





\bibitem{Lutscher}
F. Lutscher, E. Pachepsky and M.A. Lewis,
\newblock The effect of dispersal patterns on stream populations,
\newblock\emph{SIAM Rev} \textbf{47} (2005), no. (4)749--772.




\bibitem{Mogilner-E}
A. Mogilner and L. Edelstein-Keshet,
\newblock A non-local model for a swarm,
\newblock\emph{J. Math. Biol.} \textbf{38} (1999), 534--570.








\bibitem{OL-book}
A. Okubo and S. A. Levin,
\newblock\emph{Diffusion and Ecological Problems: Modern Perspectives},
\newblock Interdisciplinary Applied Mathematics, Vol. 14, 2nd ed. Springer, Berlin, 2001.







\bibitem{Skellam}
J. G. Skellam,
\newblock Random dispersal in theoretical populations,
\newblock\emph{Biometrika} \textbf{38}, (1951). 196--218.

\bibitem{Turchin}
P. Turchin,
\newblock\emph{Quantitative analysis of movement: measuring and modeling population redistribution in
animals and plants},
\newblock Sinauer Associates, Sunderland (1998).





\end{thebibliography}
\end{document}